\documentclass[12pt]{article}
\usepackage{amssymb}
\usepackage{amsthm}
\usepackage{amsmath}
\usepackage{graphicx}
\newtheorem{theorem}{Theorem}[section]
\newtheorem{definition}[theorem]{Definition}
\newtheorem{lemma}[theorem]{Lemma}

\newtheorem{corollary}[theorem]{Corollary}
\title{On a variety of commutative multiplicatively idempotent semirings\thanks{This is a post-peer-review, pre-copyedit version of an article published in Semigroup Forum. The final authenticated version is available online at \texttt{http://dx.doi.org/10.1007/s00233-016-9786-9}.}}
\author{Ivan Chajda and Helmut L\"anger$^1$}
\date{}
\begin{document}
\footnotetext[1]{Support of the research of both authors by the Austrian Science Fund (FWF), project I~1923-N25, and the Czech Science Foundation (GA\v CR), project 15-34697L, is gratefully acknowledged.}
\maketitle
\begin{abstract}
We prove that the variety ${\mathcal V}$ of commutative multiplicatively idempotent semirings satisfying $x+y+xyz\approx x+y$ is generated by single semirings. Moreover, we describe a normal form system for terms in ${\mathcal V}$ and we show that the word problem in ${\mathcal V}$ is solvable. Although ${\mathcal V}$ is locally finite, it is residually big.
\end{abstract}
 


\section{Introduction}

Semirings form a common generalization of unitary rings and bounded distributive lattices. Multiplicatively idempotent semirings were also studied in the papers \cite{PZ} and \cite{VP}. However, these investigations differ from our one since we consider semirings as algebras of a different similarity type. Semirings were successfully applied in Theoretical Computer Science, see \cite{KS}. Multiplicatively idempotent semirings were also treated in \cite{CLS}. The variety ${\mathcal C}$ of commutative multiplicatively idempotent semirings was studied by the authors in their recent paper \cite{CL}. It was shown that ${\mathcal C}$ contains linearly ordered subdirectly irreducible members of arbitrary cardinality and, moreover, subdirectly irreducible members which are not linearly ordered. This situation strongly differs from the case of so-called Boolean semirings which are commutative multiplicatively idempotent semirings satisfying the identity
\begin{equation}\label{equ5}
1+x+x\approx1.
\end{equation}
F.~Guzm\'an (\cite{Gu}) proved that the variety ${\mathcal B}$ of Boolean semirings contains just two subdirectly irreducible members, namely the two-element lattice and the two-element Boolean ring. This motivated us to find an identity similar to (\ref{equ5}) which allows to restrict the number of subdirectly irreducible members. It was recognized by the authors that among the infinitely many linearly ordered subdirectly irreducible members of ${\mathcal C}$ mentioned in \cite{CL} there is only one semiring having more than two elements and satisfying
\begin{equation}\label{equ6}
1+x+xy\approx1+x,
\end{equation}
namely the three-element so-called semiring ${\bf T}_3$. In this paper we will study the variety generated by ${\bf T}_3$. Surprisingly, this variety also contains an infinite number of subdirectly irreducible members which, however, are not linearly ordered. The semiring ${\bf T}_3$ plays an important role in some applications in three-valued logics where the connective conjunction is considered as infimum, but disjunction is not assumed. Such logics are applied in problems connected with preference tasks. Although this structure is usually considered as a set with two distinct orderings it is better understandable as a semiring.

For basic concepts on semirings the reader is referred to \cite{Go}.

\section{Basic concepts}

We start with the definition of a semiring in the sense of the monograph \cite{Go} by J.~S.~Golan.

\begin{definition}\label{def1}
A {\em semiring} is an algebra ${\bf S}=(S,+,\cdot,0,1)$ of type $(2,2,0,0)$ such that
\begin{itemize}
\item $(S,+,0)$ is a commutative monoid.
\item $(S,\cdot,1)$ is a monoid.
\item The operation $\cdot$ is distributive with respect to $+$.
\item $x0=0x=0$ for all $x\in S$
\end{itemize}
${\bf S}$ is called {\em trivial} if $|S|=1$, {\em commutative} if $\cdot$ is commutative, {\em multiplicatively idempotent} if $\cdot$ is idempotent and {\em Boolean} {\rm(}cf.\ {\rm\cite{Gu})} if it is commutative and multiplicatively idempotent and satisfies {\rm(\ref{equ5})}. Let ${\mathcal S}$ denote the variety of semirings, ${\mathcal C}$ the variety of commutative multiplicatively idempotent semirings, ${\mathcal B}$ the variety of Boolean semirings, ${\mathcal V}$ the subvariety of ${\mathcal C}$ determined by
\begin{equation}\label{equ7}
x+y+xyz\approx x+y
\end{equation}
and ${\mathcal T}$ the variety of trivial semirings.
\end{definition}

Of course, (\ref{equ7}) implies (\ref{equ6}).

It is evident that every bounded distributive lattice is a commutative multiplicatively idempotent semiring (where $+$ is join and $\cdot$ is meet). Denote by ${\mathcal D}$ the variety of bounded distributive lattices.

Let ${\bf S}=(S,+,\cdot,0,1)\in{\mathcal C}$. Since $\cdot$ is associative, commutative and idempotent, $(S,\cdot)$ forms a {\em semilattice} which we will consider as a meet-semilattice, i.~e.\ $0$ then becomes the least and $1$ the greatest element of the corresponding poset $(S,\leq)$. The semiring ${\bf S}$ is called {\em linearly ordered} if $(S,\leq)$ is a chain.

Now, we introduce our three-element semiring ${\bf T}_3$.

Let ${\bf T}_3$ denote the semiring $(\{0,a,1\},+,\cdot,0,1)$ defined by
\[
\begin{array}{c|ccc}
+ & 0 & a & 1 \\
\hline
0 & 0 & a & 1 \\
a & a & a & a \\
1 & 1 & a & a
\end{array}
\quad\mbox{and}\quad
\begin{array}{c|ccc}
\cdot & 0 & a & 1 \\
\hline
0     & 0 & 0 & 0 \\
a     & 0 & a & a \\
1     & 0 & a & 1
\end{array}
\]
and ${\bf S}_3$ the semiring coinciding with ${\bf T}_3$ with the only exception that $1+1=1$ instead of $1+1=a$. It is evident that both ${\bf T}_3$ and ${\bf S}_3$ are linearly ordered but none of them is a unitary ring or a bounded distributive lattice (or a product of such algebras since they are subdirectly irreducible).

In the following let ${\mathcal V}({\bf S})$ denote the variety generated by a given semiring ${\bf S}$. Surprisingly, ${\mathcal V}({\bf T}_3)$ turns out to be {\em finitely based} (i.~e.\ it has a finite basis of identities) and residually big. Since ${\mathcal V}({\bf T}_3)$ is generated by a finite semiring, it is locally finite (cf.\ \cite B).

Let ${\mathcal V}$ denote the subvariety of ${\mathcal C}$ determined by (\ref{equ7}). One can easily check that ${\bf T}_3\in{\mathcal V}$ and hence ${\mathcal V}({\bf T}_3)\subseteq{\mathcal V}$.

A short inspection shows that ${\bf T}_3$ can be expressed in the form ${\bf 2}\oplus1$ where ${\bf 2}$ denotes the two-element lattice $(\{0,a\},\vee,\wedge,0,a)$ and
\begin{equation}\label{equ8}
x+y:=\left\{
\begin{array}{ll}
x\vee y & \mbox{if }x,y\neq1 \\
1       & \mbox{if }(x,y)\in\{(0,1),(1,0)\} \\
a       & \mbox{otherwise}
\end{array}
\right.\mbox{ and }\;xy:=\left\{
\begin{array}l
x\wedge y \\
y \\
x
\end{array}
\right\}\mbox{ if }\left\{
\begin{array}l
x,y\neq1 \\
x=1 \\
y=1
\end{array}
\right.
\end{equation}
This motivates us to generalize this construction as follows:

\begin{definition}
Let ${\bf L}=(L,\vee,\wedge,0,a)$ be a non-trivial bounded distributive lattice, $1\notin L$ and $S:=L\cup\{1\}$ and define binary operations $+$ and $\cdot$ on $S$ according to {\rm(\ref{equ8})}. Then ${\bf L}\oplus1$ denotes the semiring $(S,+,\cdot,0,1)$.
\end{definition}

As can be easily verified, ${\bf L}\oplus1:=(S,+,\cdot,0,1)\in{\mathcal V}$ and hence ${\mathcal V}({\bf L}\oplus1)\subseteq{\mathcal V}$. Moreover, ${\bf L}\oplus1$ is subdirectly irreducible provided ${\bf L}$ is a Boolean lattice (cf.\ \cite{CL}).

${\bf 2}$ is a subdirectly irreducible member of ${\mathcal V}$ satisfying $a+a=a$.

\section{Canonical forms of terms in ${\mathcal V}$}

We are going to derive a canonical form of terms in ${\mathcal V}$. At first we describe the form of terms in ${\mathcal C}$. For this we introduce the following

\begin{definition}
Let $\mathbb N_0$ denote the set of non-negative integers and $n\in\mathbb N_0$ and put $N:=\{1,\ldots,n\}$. We define a natural linear order $\leq$ on $2^N$ by $I\leq J$ if either
\begin{itemize}
\item $I=J$ or
\item $|I|=|J|$, $I=\{i_1,\ldots,i_k\}$, $i_1<\ldots<i_k$, $J=\{j_1,\ldots,j_k\}$, $j_1<\ldots<j_k$ and there exists an $l\in\{1,\ldots,k\}$ with $(i_1,\ldots,i_{l-1})=(j_1,\ldots,j_{l-1})$ and $i_l<j_l$ or
\item $|I|<|J|$
\end{itemize}
{\rm(}$I,J\in2^N${\rm)}.
\end{definition}

Next we prove that within ${\mathcal C}$ terms can be written in some canonical form.

\begin{lemma}
Every term $t(x_1,\ldots,x_n)$ in ${\mathcal C}$ can be written in the form
\begin{equation}\label{equ1}
t(x_1,\ldots,x_n)=\sum_{r=1}^m\prod_{s\in I_r}x_s\mbox{ with }m\in\mathbb N_0,I_1,\ldots,I_m\in2^N,I_1\leq\ldots\leq I_m
\end{equation}
where the empty sum is defined as $0$ and the empty product as $1$.
\end{lemma}

\begin{proof}
It is clear that expressions of the form (\ref{equ1}) are terms in ${\mathcal C}$, that $0$ and $1$ are special cases and that sum and product of two expressions of the form (\ref{equ1}) can be written again in this form.
\end{proof}

Since
\[
{\mathcal V}\models x+x+x\approx x+x+xxx\approx x+x
\]
there exist only finitely many different terms in ${\mathcal V}$ of fixed finite arity which means that ${\mathcal V}$ is locally finite.

Within ${\mathcal V}$ we can write terms in a more economic way.

\begin{definition}
A representation of the form {\rm(\ref{equ1})} is called {\em reduced} if there do not exist mutually distinct $i,j,k\in\{1,\ldots,m\}$ with $I_i\cup I_j\subseteq I_k$.
\end{definition}

\begin{lemma}\label{lem1}
In ${\mathcal V}$ every term has a reduced representation.
\end{lemma}

\begin{proof}
Let $t(x_1,\ldots,x_n)$ be a term in ${\mathcal V}$ of the form (\ref{equ1}) and assume that the representation (\ref{equ1}) is not reduced. Then there exist mutually distinct $i,j,k\in\{1,\ldots,m\}$ with $I_i\cup I_j\subseteq I_k$. We have
\begin{eqnarray*}
{\mathcal V} & \models & \prod_{s\in I_i}x_s+\prod_{s\in I_j}x_s+\prod_{s\in I_k}x_s\approx\prod_{s\in I_i}x_s+\prod_{s\in I_j}x_s+\prod_{s\in I_i}x_s\prod_{s\in I_j}x_s\prod_{s\in I_k}x_s\approx \\
& \approx & \prod_{s\in I_i}x_s+\prod_{s\in I_j}x_s
\end{eqnarray*}
and hence
\[
{\mathcal V}\models t(x_1,\ldots,x_n)\approx\sum_{\substack{r=1 \\ r\neq k}}^m\prod_{s\in I_r}x_s
\]
with $I_1\leq\ldots\leq I_{k-1}\leq I_{k+1}\leq\ldots\leq I_m$. Either the last representation is reduced or again one summand can be cancelled. Going on in this way one finally ends up with a reduced representation within a finite number of steps.
\end{proof}

This lemma allows to enumerate all $n$-ary terms in ${\mathcal V}$ for given $n$.

\begin{corollary}
For $n=0,1,2$ we list all terms within ${\mathcal V}$ in $n$ variables: \\
$n=0:0,1,1+1$ \\
$n=1:0,1,x,1+1,1+x,x+x$ \\
$n=2:0,1,x,y,xy,1+1,1+x,1+y,1+xy,x+x,x+y,x+xy,y+y,y+xy,xy+xy,1+x+y,x+x+y,x+y+y$
\end{corollary}

\begin{proof}
The proof is evident.
\end{proof}

\section{The variety generated by ${\bf T}_3$}

In this section, we prove ${\mathcal V}({\bf S})={\mathcal V}$ for every ${\bf S}=(S,+,\cdot,0,1)\in{\mathcal V}$ with $1+1\neq0,1$, in particular ${\mathcal V}({\bf L}\oplus1)={\mathcal V}$ and hence ${\mathcal V}({\bf T}_3)={\mathcal V}$. Our crucial result is the following

\begin{theorem}\label{th1}
Let ${\bf S}=(S,+,\cdot,0,1)\in{\mathcal V}$ with $1+1\neq1$. Then, for terms $t(x_1,\ldots,x_n)$ and $u(x_1,\ldots,x_n)$ in ${\mathcal V}$
\begin{equation}\label{equ2}
{\bf S}\models t(x_1,\ldots,x_n)\approx u(x_1,\ldots,x_n)
\end{equation}
implies
\begin{equation}\label{equ3}
{\mathcal V}\models t(x_1,\ldots,x_n)\approx u(x_1,\ldots,x_n)
\end{equation}
\end{theorem}

\begin{proof}
Since $1+1=0$ would imply
\[
0=1+1=1+1+1\cdot1\cdot1=1+1+1=0+1=1,
\]
we have $1+1\neq0$. Hence $0,1,1+1$ are mutually distinct. Let $t(x_1,\ldots,x_n)$ and $u(x_1,\ldots,x_n)$ be terms in ${\mathcal C}$ satisfying (\ref{equ2}). According to Lemma~\ref{lem1} there exist corresponding reduced representations $t_1(x_1,\ldots,x_n)$ and $u_1(x_1,\ldots$ $\ldots,x_n)$, say
\begin{eqnarray*}
& & t_1(x_1,\ldots,x_n)=\sum_{r=1}^v\prod_{s\in I_r}x_s\mbox{ with }v\in\mathbb N_0,I_1,\ldots,I_v\in2^N,I_1\leq\ldots\leq I_v\mbox{ and} \\
& & u_1(x_1,\ldots,x_n)=\sum_{r=1}^w\prod_{s\in J_r}x_s\mbox{ with }w\in\mathbb N_0,J_1,\ldots,J_w\in2^N,J_1\leq\ldots\leq J_w.
\end{eqnarray*}
We have
\begin{equation}\label{equ4}
{\bf S}\models t_1(x_1,\ldots,x_n)\approx u_1(x_1,\ldots,x_n).
\end{equation}
Without loss of generality assume $v\leq w$. For every $I\in2^N$ let $\vec a_I$ denote the element $(a_1,\ldots,a_n)$ of $S^n$ satisfying $a_i=1$ for all $i\in I$ and $a_i=0$ otherwise. Suppose $(I_1,\ldots,I_v)\neq(J_1,\ldots,J_w)$. We distinguish the following cases: \\
Case~1. There exists a $z\in\{1,\ldots,v\}$ with $(I_1,\ldots,I_{z-1})=(J_1,\ldots,J_{z-1})$ and $I_z\neq J_z$. \\
Case~1.1. $I_z<J_z$. \\
Case~1.1.1. There exists an $r\in\{1,\ldots,z-1\}$ with $I_r\subseteq I_z$. \\
Then we have $t_1(\vec a_{I_z})=1+1\neq1=u_1(\vec a_{I_z})$ contradicting (\ref{equ4}). \\
Case~1.1.2. There exists no $r\in\{1,\ldots,z-1\}$ with $I_r\subseteq I_z$. \\
Then we have $t_1(\vec a_{I_z})\neq0=u_1(\vec a_{I_z})$ contradicting (\ref{equ4}). \\
Case~1.2. $I_z>J_z$. \\
This case can be treated analogously to Case~1.1 by interchanging $I_z$ and $J_z$. \\
Case~2. $(I_1,\ldots,I_v)=(J_1,\ldots,J_v)$. \\
Then $v<w$. \\
Case~2.1. There exists an $r\in\{1,\ldots,v\}$ with $I_r\subseteq J_{v+1}$. \\
Then we have $t_1(\vec a_{J_{v+1}})=1\neq1+1=u_1(\vec a_{J_{v+1}})$ contradicting (\ref{equ4}). \\
Case~2.2. There exists no $r\in\{1,\ldots,v\}$ with $I_r\subseteq J_{v+1}$. \\
Then we have $t_1(\vec a_{J_{v+1}})=0\neq u_1(\vec a_{J_{v+1}})$ contradicting (\ref{equ4}). \\
Hence $(I_1,\ldots,I_v)=(J_1,\ldots,J_w)$. This shows that the representations \\
$t_1(x_1,\ldots,x_n)$ and $u_1(x_1,\ldots,x_n)$ coincide. Now (\ref{equ3}) follows from
\[
{\mathcal V}\models t(x_1,\ldots,x_n)\approx t_1(x_1,\ldots,x_n)\mbox{ and }{\mathcal V}\models u(x_1,\ldots,x_n)\approx u_1(x_1,\ldots,x_n).
\]
\end{proof}

From the proof of Theorem~\ref{th1} if follows that the reduced representation of a term in ${\mathcal V}$ is unique and hence such representations constitute a normal form system for terms in ${\mathcal V}$.

Since ${\bf S}\in{\mathcal V}$ we conclude ${\mathcal V}({\bf S})\subseteq{\mathcal V}$ which together with Theorem~\ref{th1} yields

\begin{corollary}
For every ${\bf S}=(S,+,\cdot,0,1)\in{\mathcal V}$ with $1+1\neq1$ we have ${\mathcal V}({\bf S})={\mathcal V}$, in particular ${\mathcal V}({\bf L}\oplus1)={\mathcal V}$ and ${\mathcal V}({\bf T}_3)={\mathcal V}$.
\end{corollary}

From this we conclude important structural properties of ${\mathcal V}$.

\begin{corollary}
$\mbox{}$
\begin{itemize}
\item The variety ${\mathcal V}$ is locally finite.
\item For each positive integer $n$ the variety ${\mathcal V}$ has a subdirectly irreducible member of cardinality $2^n+1$.
\item For each infinite cardinal $k$ the variety ${\mathcal V}$ has a subdirectly irreducible member of cardinality $k$.
\item ${\mathcal V}$ is residually big.
\item ${\mathcal V}$ has a normal form system for its terms and the word problem in ${\mathcal V}$ is solvable.
\end{itemize}
\end{corollary}

The previous results reveal the prominency of ${\bf T}_3$ since ${\mathcal V}={\mathcal V}({\bf T}_3)$. Although (\ref{equ7}) implies
\begin{equation}
1+x+x\approx1+x
\end{equation}
which looks similar to (\ref{equ5}), the varieties ${\mathcal B}$ and ${\mathcal V}$ show completely different structural behaviour.

\section{Some subvarieties of ${\mathcal C}$}

Several subvarieties of ${\mathcal C}$ were already mentioned. In the following theorem, we show how ${\mathcal B}$ and ${\mathcal V}$ are located within the lattice of subvarieties of ${\mathcal S}$.

\begin{theorem}
All inclusions in the Hasse diagram
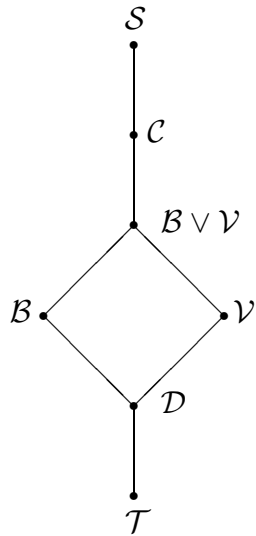
\begin{center}
\setlength{\unitlength}{6mm}
\begin{picture}(6,12)
\put(3,1){\circle*{.2}}
\put(3,3){\circle*{.2}}
\put(1,5){\circle*{.2}}
\put(5,5){\circle*{.2}}
\put(3,7){\circle*{.2}}
\put(3,9){\circle*{.2}}
\put(3,11){\circle*{.2}}
\put(3,1){\line(0,1)2}
\put(3,7){\line(0,1)4}
\put(1,5){\line(1,-1)2}
\put(1,5){\line(1,1)2}
\put(5,5){\line(-1,-1)2}
\put(5,5){\line(-1,1)2}
\put(2.8,.2){${\mathcal T}$}
\put(3.6,2.85){${\mathcal D}$}
\put(.25,4.85){${\mathcal B}$}
\put(5.2,4.85){${\mathcal V}$}
\put(3.6,6.85){${\mathcal B}\vee{\mathcal V}$}
\put(3.3,8.85){${\mathcal C}$}
\put(2.8,11.35){${\mathcal S}$}
\end{picture}
\end{center}
are proper. Moreover, ${\mathcal B}\cap{\mathcal V}$ is the variety ${\mathcal D}$ of bounded distributive lattices.
\end{theorem}

\begin{proof}
$\mbox{}$
\begin{itemize}
\item The two-element lattice belongs to $({\mathcal B}\cap{\mathcal V})\setminus{\mathcal T}$.
\item ${\rm GF(2)}\in({\mathcal B}\setminus({\mathcal B}\cap{\mathcal V}))\cap(({\mathcal B}\vee{\mathcal V})\setminus{\mathcal V})$
\item ${\bf T}_3\in({\mathcal V}\setminus({\mathcal B}\cap{\mathcal V}))\cap(({\mathcal B}\vee{\mathcal V})\setminus{\mathcal B})$
\item ${\rm GF(3)}\in{\mathcal S}\setminus{\mathcal C}$
\item ${\bf S}_3\in{\mathcal C}\setminus({\mathcal B}\vee{\mathcal V})$ since
\[
{\mathcal B}\models1+x+xy+xy\approx1+x(1+y+y)\approx1+x\cdot1\approx1+x
\]
and
\[
{\mathcal V}\models1+x+xy+xy\approx1+x
\]
and hence
\[
{\mathcal B}\vee{\mathcal V}={\rm HSP}({\mathcal B}\cup{\mathcal V})\models1+x+xy+xy\approx1+x,
\]
but in ${\bf S}_3$ we have
\[
1+1+1\cdot a+1\cdot a=1+1+a+a=1+a=a\neq1=1+1.
\]
\end{itemize}
Since ${\mathcal B}$ has only two subdirectly irreducible members, namely the two-element Boolean ring ${\bf R}=(R,+,\cdot,0,1)$ and ${\bf 2}$, and since ${\bf R}$ satisfies $1+1=0$ and hence ${\bf R}\notin{\mathcal V}$, we conclude that ${\mathcal B}\cap{\mathcal V}$ contains just one subdirectly irreducible member, namely ${\bf 2}$, and hence it coincides with ${\mathcal D}$.
\end{proof}

Authors' addresses:

Ivan Chajda \\
Palack\'y University Olomouc \\
Faculty of Science \\
Department of Algebra and Geometry \\
17.\ listopadu 12 \\
771 46 Olomouc \\
Czech Republic \\
ivan.chajda@upol.cz

Helmut L\"anger \\
TU Wien \\
Faculty of Mathematics and Geoinformation \\
Institute of Discrete Mathematics and Geometry \\
Wiedner Hauptstra\ss e 8-10 \\
1040 Vienna \\
Austria \\
helmut.laenger@tuwien.ac.at
\end{document}